\documentclass{article}
\usepackage{amsthm,amsmath,amssymb,amsfonts,latexsym,enumerate,vmargin}
\usepackage[usenames,dvipsnames]{color}
\usepackage{hyperref}
\usepackage{tabu}
\usepackage{tikz}
\usetikzlibrary{arrows}

\setmarginsrb{30mm}{20mm}{30mm}{20mm}{0pt}{10mm}{0pt}{10mm}

\def\F{\mathbb{F}}

\newtheorem{thm}{Theorem}
\newtheorem*{thm*}{Theorem}

\newtheorem*{claim*}{Claim}

\newtheorem*{dfn*}{Definition}

\newtheorem*{ntn*}{Notation}
\newtheorem{lem}[thm]{Lemma}
\newtheorem*{lem*}{Lemma}
\newtheorem{prop}[thm]{Proposition}
\newtheorem*{prop*}{Proposition}
\newtheorem{cor}[thm]{Corollary}
\newtheorem*{cor*}{Corollary}

\newtheorem*{conj*}{Conjecture}
\theoremstyle{remark}

\newtheorem*{rmk*}{Remark}

\newtheorem*{rmks*}{Remarks}

\newcommand{\squaregraph}[1]
{\tikzstyle{every node}=[circle, draw, fill=black!50, inner sep=0pt, minimum width=4pt]
\begin{tikzpicture}[thick,baseline=-4pt]
\node at(-0.2,0.2)(0){}; \node at(-0.2,-0.2)(1){};
\node at (0.2,0.2)(2){}; \node at (0.2,-0.2)(3){};
\draw { #1 };
\end{tikzpicture}}

\newcommand{\pentagraph}[1]
{\tikzstyle{every node}=[circle, draw, fill=black!50, inner sep=0pt, minimum width=3pt]
\begin{tikzpicture}[thick,baseline=-4pt]
\node at (-0.2,-0.25)(1){};
\node at (-0.2, 0.05)(2){};
\node at ( 0.0, 0.25)(3){};
\node at ( 0.2, 0.05)(4){};
\node at ( 0.2,-0.25)(5){};
\draw { #1 };
\end{tikzpicture}}

\title{A Note on the Inducibility of $4$-vertex Graphs}
\author{
Chaim Even-Zohar\thanks{
Department of Mathematics, Hebrew University, Jerusalem 91904, Israel. \newline
e-mail: \href{mailto:chaim.evenzohar@mail.huji.ac.il}{chaim.evenzohar@mail.huji.ac.il}~.}
\and
Nati Linial\thanks{
Department of Computer Science, Hebrew University, Jerusalem 91904, Israel.\newline
e-mail: \href{mailto:nati@cs.huji.ac.il}{nati@cs.huji.ac.il}~.
Supported by grants from the ERC and from the ISF.}
}

\begin{document}

\maketitle

\begin{abstract}

There is much recent interest in understanding the density at which
constant size graphs can appear in a very large graph.
Specifically, the inducibility of a graph $H$ is its extremal density,
as an induced subgraph of $G$, where $|G| \rightarrow \infty$.
Already for $4$-vertex graphs many questions are still open.
Thus, the inducibility of the $4$-path
was addressed in a construction of Exoo (1986), but remains unknown.
Refuting a conjecture of Erd\H{o}s,
Thomason (1997) constructed graphs with a small density of both $4$-cliques and $4$-anticliques.
In this note, we merge these two approaches and construct better graphs for both problems.
\end{abstract}

\section{Introduction}\label{introsect}

Let $H$ and $G$ be simple graphs.
Denote by $P(H,G)$ the proportion of $H$ in~$G$, i.e.,
the number of induced copies of $H$ in $G$, divided by {\small$\binom{|G|}{|H|}$}.
More generally, consider the \emph{local $t$-profile} of $G$,
the vector $\mathbf{P}_t(G) = \{P(H,G)\}_H$, where
$H$ runs over all isomorphism types of $t$-vertex graphs.

It is a major challenge to understand the limit points of $\mathbf{P}_t(G)$ as~$|G| \rightarrow \infty$, 
since in full generality this question includes large portions of extremal graph theory.
An important example is the study of graph inducibility, that was started by Pippenger and Golumbic~\cite{pippenger1975}.
To the best of our knowledge, the general concept of a local profile was essentially first considered in~\cite{erdos1979}.
More recently, the study of graph limits~\cite{lovasz2012} has brought to the fore the significance of local $t$-profiles of large graphs. 
The relevance of these concepts in the computational realm is illustrated by the study of graph property testing~\cite{goldreich1998}.
A key advance that enabled much of the recent progress in this area is Razborov's theory of flag algebras~\cite{razborov2007,razborov2013}.
Indeed, there is substantial recent activity in this domain~\cite{huang2014, huang2012, linial2014graphs},
and additional combinatorial structures with natural notions of local profile and inducibility are being investigated as well, 
e.g. tournaments~\cite{linial2014tournaments} trees~\cite{bubeck2013}, and permutations~\cite{wilf2002}.

The \emph{(maximal) inducibility} of a graph $H$ is
$$I(H) = \lim\limits_{n \rightarrow \infty} \max\limits_{|G|=n} P(H,G).$$
An averaging argument shows that the sequence decreasingly converges to a limit,
and the implicit error term is bounded by $|H|^2/|G|$ (\cite{pippenger1975}, Theorem $3$).
We briefly review a few facts and figures about inducibility.

The inducibility of some graphs is known precisely.
Clearly cliques and anticliques have inducibility $1$.
In general, by passing from $G$ to its complement $G^C$
one can easily obtain $I(H) = I(H^C)$.
The inducibility of complete partite graphs has been considered in the literature
\cite{pippenger1975,exoo1986,bollobas1986,brown1994,bollobas1995}. 
For example, the inducibility of the complete bipartite graph $K_{t,t}$, is $I(K_{t,t}) = \binom{2t}{t}/4^t \approx 1/\sqrt{\pi t}$,
as attained by larger bipartite graphs.

Every $t$-vertex graph $H$ has inducibility $\ge t!/(t^t-t) \approx \sqrt{2 \pi t} \exp(-t)$,
as shown by a nested blow-up of $H$
(\cite{pippenger1975}, see Section~\ref{constsect}).
Sometimes, one can exploit symmetries of $H$,
and adjust this construction to yield better lower bounds~\cite{siran1984}.
However, generically the bound $\exp(-t)$ is nearly tight.
Almost every~$H$ contains a set $S$ of about $3\log t$ vertices
that ``separates'' between the rest of the vertices.
Namely for every two vertices $x, y \in V(H) \setminus S$
there holds $\Gamma(x)\cap S\neq\Gamma(y)\cap S$, where $\Gamma(v)$ is $v$'s set of neighbors.
This implies that $I(H) \leq \exp( - t + O(\log^2 t))$.

Can we determine the exact inducibility of small graphs?
For $3$ vertices, we only need to consider $K_{1,2}$ or its complement.
Large bipartite graphs yield $I(K_{1,2}) \geq 3/4$,
which is optimal by the following classic result.
\begin{thm*}[Goodman~\cite{goodman1959}]
In a graph with $2n$ vertices, at least $2\binom{n}{3}$
triples form a triangle or an anti-triangle.
\end{thm*}

There are eleven isomorphism types of $4$-vertex graphs,
the inducibilities of ten of which are known, and summarized in Table~\ref{exootable},
an updated version of Exoo's~\cite{exoo1986}.
Some of these numbers follow from general results concerning complete bipartite
graphs~\cite{pippenger1975,bollobas1986, brown1994, bollobas1995}.
For others, various extremal constructions were found,
and their optimality was proved using flag algebra~\cite{hirst2011}.

\begin{table}[ht]
\centering\renewcommand{\arraystretch}{2}\begin{tabu}{l l l l c c l}
\tabucline[1pt]{-}\tabucline[1pt]{-}
$H$ & & $H^C$ & & $I(H)$ & & Extremal Construction \\
\tabucline[1pt]{-}\tabucline[1pt]{-}
$K_4$ &\squaregraph{(0)--(1)  (0)--(2)  (0)--(3)  (1)--(2)  (2)--(3)  (3)--(1)} &
$A_4$ &\squaregraph{} &
1 & & A complete graph \\
\hline
$S_4$ &\squaregraph{(0)--(1)  (0)--(2)  (0)--(3)} $\;\;\;\;$&
$T_4$ &\squaregraph{(1)--(2)  (2)--(3)  (3)--(1)} $\;\;\;\;$&
1/2 & & A complete bipartite graph \\
\hline
$C_4$ &\squaregraph{(0)--(1)  (1)--(3)  (3)--(2)  (2)--(0)} &
$M_4$ &\squaregraph{(1)--(2)  (3)--(0)} &
3/8 & & A complete bipartite graph \\
\hline
$V_4$ &\squaregraph{(3)--(1)  (3)--(2)  } &
$Q_4$ &\squaregraph{(1)--(2)  (2)--(0)  (0)--(1)  (3)--(0)} &
3/8 & & Two disjoint complete bipartite graphs \\
\hline
$D_4$ &\squaregraph{(0)--(1)  (0)--(2)  (0)--(3)  (2)--(3)  (3)--(1)} &
$E_4$ &\squaregraph{(1)--(2)  } &
72/125 & & A complete $5$-partite graph \\
\hline
$P_4$ &\squaregraph{(0)--(1)  (1)--(2)  (2)--(3)} &
& & ? & & Unknown \\
\tabucline[1pt]{-}\tabucline[1pt]{-}
\end{tabu}
\caption{Inducibilities of graphs on 4 vertices}
\label{exootable}
\end{table}

The case of $I(P_4)$ is intriguing.
Exoo's construction gives a lower bound of
$60/307 \approx 0.1954$ (See Section~\ref{constsect}).
We comment that~\cite{exoo1986} and~\cite{hirst2011} state the inaccurate bound
$960/4877 \approx 0.1968$, presumably by mistake.
Upper bounds on $I(P_4)$ were first given by Exoo~\cite{exoo1986},
then by Hirst~\cite{hirst2011}, and finally Vaughan~\cite{flagmatic}
set the current record at $0.204513$, using flag algebra calculus.
Here we present a new construction, which implies:

\begin{prop}\label{IP4}
$ I(P_4) \geq 1173/5824 \approx 0.2014$
\end{prop}

We proceed with some additional definitions.
Inducibility naturally extends to formal linear combinations of graphs,
also known as \emph{quantum graphs}.
Namely, we interpret $P(H+ \nobreak H',G)$ as $P(H,G) + P(H',G)$,
and define the inducibility $I(H+H')$ accordingly.
The \emph{minimal inducibility} can be defined as the former,
$\max$ replaced by $\min$:
$$i(H) = \lim\limits_{n \rightarrow \infty} \min\limits_{|G|=n} P(H,G) .$$

By Ramsey's theorem a monochromatic $K_t$ exists
in every red/blue edge-coloring of a large enough $K_n$.
Goodman~\cite{goodman1959} and Erd\H{o}s~\cite{erdos1962}
asked how many monochromatic $K_t$-s such a coloring must contain.
The asymptotic nature of this quantity is
captured in the minimal inducibility $i(K_t+A_t)$, where $A_t = K_t^C$ is an anticlique on $t$ vertices.
By the above theorem of Goodman, $i(K_3+A_3) = 1/4$.
Interestingly, this is attained not only by bipartite graphs but also by random ones.
Using Ramsey numbers on the one hand and random graphs on the other,
Erd\H{o}s~\cite{erdos1962} showed
$$ \frac{1}{\binom{\binom{2t-2}{t-1}}{t}}
\;\leq\; i(K_t+A_t) \;\leq\; \frac{2}{2^{\binom{t}{2}}} \;,$$
and added that the upper bound seems likely to be tight.

This conjecture was refuted for all $t \geq 4$ by Thomason
(\cite{thomason1989}, see also~\cite{jagger1996,thomason1997}).
For $t= \nobreak 4$, Thomason showed that
$i(K_4+A_4)$ is \emph{strictly} smaller than $3769/124416 \approx 1/33.0103$.
His construction will be described below.
Another family of counterexamples for $4 \leq t \leq 8$ was started by
Franek and R\"{o}dl~\cite{franek1993,franek2002,deza2012,shen2012}
(see Section~\ref{discsect}).
Lower bounds on $i(K_4+A_4)$ were obtained by
several authors~\cite{giraud1979,evans1981,wolf2010,sperfeld2011,niess2012,flagmatic},
with Vaughan's $0.0294343 \approx 1/33.9739$ being the current record.
Here we slightly improve Thomason's bound:

\begin{prop}\label{iK4A4}
$ i(K_4+A_4) \leq 1411/46592 \approx 1/33.0205$
\end{prop}

\emph{Overview:}
Propositions~\ref{IP4} and~\ref{iK4A4} are proved by construction of graph sequences.
In the next section we survey several constructive tools,
which provide the framework for the subsequent presentation of the graphs.
Then we prove in Section~\ref{compsect}, that
these sequences indeed approach the densities prescribed in the propositions.
In Section~\ref{discsect} we discuss the inducibility of
longer paths, larger monochromatic cliques and all $5$-vertex graphs. 

\section{Construction}\label{constsect}

We start with a technical reduction. Let $\mathbf{P}_t(G)$ be the $t$-profile of a graph $G$.
The probability that $t$ random vertices induce a copy of $H$ is $P(H,G)$.
It is useful to define a similar number, $R(H,G)$,
the same probability when the $t$ vertices are sampled with replacements.
The vector $\mathbf{R}_t(G) = \{R(H,G)\}_H$ is called the \emph{repetitive $t$-profile} of $G$.
Note that $R(H,G)$ may depend on whether there is an edge from some vertex of $G$ to itself.
Therefore, mention of $R(H,G)$ will hereafter imply that $G$ allows loops.
Note also that $R(H,G)$ can replace $P(H,G)$ in the definition of inducibility,
because in large graphs repeated sampling of vertices gets rare.

Several extremal constructions for inducibility fall under the following simple definition.
Let $G$ be a graph, with or without loops.
A \emph{blow-up} of $G$ of order $m$, is a graph on the vertex set $V(G) \times \{1,...,m\}$,
with edges given by
$$ (g,h) \sim (g',h') \;\;\;\;\; \Leftrightarrow \;\;\;\;\; g \sim g' \;.$$
For example, a complete graph is a blow-up of a loop
and a complete bipartite graph is a blow-up of an edge.
The reader may verify that if $G'$ is a blow-up of $G$ then
$\mathbf{R}_t(G') = \mathbf{R}_t(G)$ for all $t$.
Thus a sequence of blow-ups of $G$ yields,

\begin{lem}\label{qlemma}
For every two graphs $G, H$
$$ i(H) \leq R(H,G) \leq I(H) \;.$$
\end{lem}

Each of the exactly known inducibilities in Table~\ref{exootable} is
attained by a sequence of blow-ups of some graph,
but as the next lemma shows this is not always the case.
A graph is called \emph{twin-free} if no two vertices have
the same set of neighbors among the other vertices. Of course, most graphs are twin-free.
An example of relevance here is the $4$-path $P_4$.

\begin{lem}\label{twins}
For every twin-free graph $H$ and for every graph $G$,
$$ R(H,G) < I(H) \;.$$
\end{lem}

\begin{proof}
Let $G'$ be a blow-up of $G$ of order $t = |H|$.
If $U = \{(u,1),...,(u,t)\}$ is the blow-up of a vertex $u$ of $G$,
then $U$ induces either a clique or an anticlique in $G'$.
No copy of $H$ in $G'$ can have more than one vertex from $U$,
since $H$ is twin-free and every two vertices in $U$ are twins in $G'$.
We next define $G''$:
Start from $G'$ and modify the edges among $U$'s vertices to create there a new copy of $H$.
Note that this cannot jeopardize the existing induced copies of $H$.
Therefore $R(H,G) = R(H,G') < R(H,G'') \leq I(H)$ as required.
\end{proof}

\begin{rmk*}
See~\cite{hatami2011} for conditions under which inducibility is attained by blow-ups.
\end{rmk*}

Lemma~\ref{twins}'s proof naturally leads to the following construction,
by Pippenger and Golumbic~\cite{pippenger1975}.
The \emph{composition} of two graphs $G$ and $H$, which we denote $G \odot H$,
is essentially a blow-up of $G$, with a copy of $H$ corresponding to every vertex of $G$.
Formally, $G \odot H$ is a graph on $V(G) \times V(H)$, with edges
$$ (g,h) \sim (g',h') \;\;\;\;\; \Leftrightarrow \;\;\;\;\;
(g \sim g') \text{ or } (g = g' \text{ and } h \sim h') \;.$$
Note that this operation is non-commutative but associative.
We use the shorthand $G^{\odot n}$ for the iterated composition of $n$ copies of $G$.
The sequence $\left\{ G^{\odot n} \right\}_{n=1}^{\infty}$
is called the \emph{nested blow-up} of $G$.

The nested blow-up of a graph $H$ yields the above-mentioned lower bound
$I(H) \geq t!/(t^t-\nobreak t)$, where~$t=|H|$.
For $H = C_t$ $(t \geq 5)$,
Pippenger and Golumbic showed that this bound is tight up to a constant factor,
and conjectured it to be sharp.
For example, by nested blow-up
$I(C_5) \geq\nobreak 5!/(5^5- \nobreak 5) = 1/26 \approx 0.03846154$,
rather than $5!/5^5 = 0.0384$ obtained by regular blow-ups.
Flag algebra calculus yields a different but close upper bound of $0.03846157$~\cite{flagmatic}.

Exoo~\cite{exoo1986} noticed that
sometimes the nested blow-up of some $H' \neq H$ beats the nested blow-up of $H$.
Indeed, suppose that a $t$-vertex graph $H$ is obtained from a vertex-transitive graph $H'$
by removal of a single vertex.
In this case, the nested blow-up of $H'$ gives
$$ I(H) \;\geq\; \frac{t!}{(t+1)^{t-1}-1}
\;\approx\; \sqrt{2\pi} \cdot (t+1)^{3/2}e^{-t-1} ,$$
which is superior to the above bound by about $(t+1)/e$.
In particular, taking $H'=C_5$ he obtained $I(P_4) \geq 6/31 \approx 0.1935$.
By computer search he found an even better candidate, the Paley graph $Q_{17}$
in which two elements of the finite field $\F_{17}$ are neighbors iff their difference is a square.
The nested blow-up of $Q_{17}$ yields $I(P_4) \geq 60/307 \approx 0.1954$.

We turn to describe the other main building block,
which is due to Thomason~\cite{thomason1997}.
Let $G$ and $H$ be graphs, possibly with loops.
Their \emph{tensor product}, $G \otimes H$,
has vertex set $V(G) \times V(H)$, where
$$ (g,h) \sim (g',h') \;\;\;\;\; \Leftrightarrow \;\;\;\;\;
(g \sim g' \text{ and } h \not \sim h') \text{ or }
(g \not \sim g' \text{ and } h \sim h') \;.$$
In particular, $G \otimes A_n$ is a blow-up of $G$.
A less trivial example is
$K_3 \otimes K_3$  which is isomorphic to the Paley graph of order~$9$.
Note that the tensor product is associative and commutative.

As Thomason observed, and we explain below,
the $t$-profile of $G \otimes H$ can be easily computed from those of $G$ and~$H$.
This considerably simplified his original refutation of the Erd\H{o}s conjecture.
For the special case of $i(K_4 + A_4)$,
he found plenty of further counterexamples.
In particular, a computer investigation of products of small graphs yielded
$$ R(K_4 + A_4,\; M_4 \otimes K_4 \otimes K_3 \otimes K_3 ) \;=\;
\frac{11411}{373248} \;\approx\; \frac{1}{32.7095} \;.$$
Then, a broader search on larger Cayley graphs enhanced the bound to
$$ R(K_4 + A_4,\; M_4 \otimes K_4 \otimes G_{18} ) \;=\;
\frac{3769}{124416} \;\approx\; \frac{1}{33.0103} \;,$$
where we define $G_{18} = (K_3 \otimes K_3) \odot K_2$.
A final remark in~\cite{thomason1997}, attributed to a referee,
states that this construction is still not optimal.
Replacing $K_2$ by a randomly perturbed blow-up of $K_2$
yields a slight improvement ($< 10^{-7}$).

The above chain of events makes it natural to ask, what would happen
if $(K_3 \otimes K_3) \odot K_2$ were replaced by a nested blow-up of $K_3 \otimes K_3$.
Indeed, as we explain in the next section
\begin{equation}
R\left(K_4 + A_4,\; M_4 \otimes K_4 \otimes (K_3 \otimes K_3)^{\odot n} \right)
\;\xrightarrow{\;n\rightarrow\infty\;}\; \frac{1411}{46592} \;\approx\; \frac{1}{33.0205} \;.
\label{RK4A4}
\end{equation}
Note that Proposition~\ref{iK4A4} follows from this by Lemma~\ref{qlemma}.
In fact, the graph sequence that demonstrates Proposition~\ref{IP4}
hides in the $4$-profiles along the way:
\begin{equation}
I(P_4) \;\geq\; R\left(P_4,\; K_4 \otimes (K_3 \otimes K_3)^{\odot n} \right)
\;\xrightarrow{\;n\rightarrow\infty\;}\; \frac{1173}{5824} \;\approx\; 0.2014 \;.
\label{RP4}
\end{equation}

\section{Computation}\label{compsect}

We next develop general tools,
which may also provide some insight on the construction.
We begin with a slightly different formulation of Thomason's analysis
from~\cite{thomason1997}.

Recall that $\mathbf{R}_t(G)$ and $\mathbf{P}_t(G)$
are the distributions of induced graphs on $t$ random vertices,
sampled with and without replacements, respectively.
It is convenient to define the \emph{labeled $t$\nobreakdash-profile},
which accounts also for the ordering of the sample.
For a labeled $t$-vertex graph~$H$ we define $r(H,G) = R(H,G)/|orbit(H)|$,
where $orbit(H)$ corresponds to the action of the symmetric group $S_t$
through relabeling of the vertices.
Equivalently $R(H,G) = r(H,G) \cdot [S_t:Aut(H)]$,
where $Aut(H)$ is the group of automorphisms of~$H$.
To complete the picture,
we similarly define $p(H,G) = P(H,G)/|orbit(H)|$ in the non-repetitive case.

The labeled graphs of order $t$ admit a natural group structure,
with the operation of symmetric difference of the edges.
One can apply the discrete Fourier transform over this group,
and obtain $\mathbf{\hat{r}}_t(G)$, the \emph{spectral $t$-profile}:
$$ \hat{r}(H,G) = \sum\limits_{H'}(-1)^{e(H \cap H')} \cdot r(H',G) $$
The summation is over all $2^{\binom{t}{2}}$ labeled graphs with $t$ vertices,
and $e(H \cap H')$ is the number of edges which appear in both $H$ and $H'$.
Observe that $\hat{r}(A_t,G) = 1$ for every $G$.
Note also that $\hat{r}(H,G)$ is constant on classes of graph isomorphism,
exactly like $r(H,G)$.
Therefore we can refer to, e.g., $\hat{r}(C_4,G)$
without specifying labels on the vertices of $C_4$.

The basic observation about the tensor product is that
sampling vertices in $G \otimes G'$ can be separated,
in a sense, to sampling each factor independently.
More precisely,
if the $G$ components of the sample create $H$ and the $G'$ components create $H'$,
then the sampled graph is the symmetric difference $H \triangle H'$.
Consequently,
$$ \mathbf{r}_t(G \otimes G') \;=\; \mathbf{r}_t(G) \ast \mathbf{r}_t(G') $$
where $\ast$ stands for convolution over the group of labeled graphs.
From this immediately follows,

\begin{cor}[Thomason]\label{pointwise}
\begin{equation}
\hat{r}(H,\;G \otimes G') \;=\; \hat{r}(H,G) \; \hat{r}(H,G')
\label{pointwisewise}
\end{equation}
\end{cor}

We are now in position to explain Thomason's method.
The spectral $4$-profiles of many tensor products of small graphs
were computed fairly easily by Corollary~\ref{pointwise}.
By the inverse transform, $r(K_t + A_t, G)$ is the average of $\mathbf{\hat{r}}_t(G)$
over all graphs of even number of edges.
If $t=4$ then this amounts to the following linear functional:
\begin{equation}
r(K_4 + A_4, G) = \frac{
1 + \hat{r}(K_4,G) + 3\hat{r}(M_4,G) + 3\hat{r}(C_4,G) + 12\hat{r}(Q_4,G) + 12\hat{r}(V_4,G)
}{32} \;.
\label{funcK4A4}
\end{equation}
We are interested also in
\begin{equation}
r(P_4, G) = \frac{
1 - \hat{r}(K_4,G) + \hat{r}(M_4,G) - \hat{r}(C_4,G) + 4\hat{r}(Q_4,G) - 4\hat{r}(V_4,G)
}{64} \;.
\label{funcP4}
\end{equation}

The following lemma applies a similar line of reasoning,
in order to find the $t$-profile of the composition of graphs.

\begin{lem}\label{bilinear}
Let $s,t \geq 2$, and let $G$ and $G'$ be two graphs where $|G| = s$.
Then, there exists a bilinear operator $B_{s,t}$ such that
$$\mathbf{r}_t(G \odot G') \;=\; B_{s,t}(\mathbf{p}_m(G),\mathbf{r}_t(G')) $$
where $m = \min(s,t)$.
\end{lem}

\begin{proof}
The proof is a matter of straightforward computation of $B_{s,t}$,
which we shortly spell out
for completeness of exposition and for future reference.
All the necessary terminology is developed,
but some easy verifications are left to the reader.

Suppose that $H$ is a labeled graph on the vertex set $V = V(H)$.
The induced labeled graph on a vertex subset $V' \subseteq V$ is denoted $H[V']$.
In the other direction, if $H'$ is a labeled graph on some $V' \subseteq V$,
then we denote $\Gamma(H',V) = \{H : V(H)=V,\;H[V'] = H'\}$.
The following observation is obvious:
\begin{equation}
p(H',G) \;=\; \sum\limits_{H \in \Gamma(H',V)} p(H,G) \;.
\label{ppp}
\end{equation}
In other words, for every $j < m$ the $j$-profile $\mathbf{p}_j(G)$ is given by
a linear projection of $\mathbf{p}_m(G)$ on $j$ fixed vertices.
This relation will permit us some flexibility when we work with
the first argument of $B_{s,t}$.

Let $\Lambda(V)$ be the collection of all partitions of $V$ into disjoint sets.
Given a partition $\lambda \in \Lambda(V)$,
we denote the number of the parts by $\ell = \ell(\lambda)$,
and arbitrarily fix their indexes:
$ V = \lambda_1 \cup \lambda_2 \cup ... \cup \lambda_{\ell} $.
By abuse of notation, we write
$ \Gamma(H,\lambda) = \{H' : V(H')=V(H), \;\forall i\; H'[\lambda_i] = H[\lambda_i] \} $.
For example, if $\lambda = \{V\}$, the trivial one-part partition,
then $\Gamma(H,\lambda) = \{H\}$,
while the other trivial partition into singletons $\lambda = \{\{v\}:v\in V\}$
yields all graphs on $V$.

The \emph{transversals} of $\lambda$ are the ordered sets of representatives,
one from each part:
$ \text{Tr}(\lambda) \;=\;
\left\{ (v_1,v_2,...,v_{\ell}) : \forall i\; v_i \in \lambda_i \right\} $.
A partition $\lambda \in \Lambda(V)$ is said to be \emph{admissible} with respect to $H$
if $H[V'] = H[V'']$ for every $V',V'' \in \text{Tr}(\lambda)$,
where $V'$ and $V''$ are naturally identified according to the partition.
For example, the two above-mentioned trivial partitions are $H$-admissible for every~$H$.
The reader may verify that if $H=P_4$ then no other partition is admissible.
The set of all $H$-admissible partitions in denoted~$\Lambda(H)$.
By further abuse of notation, if $\lambda \in \Lambda(H)$
we define $H[\lambda] = H[V']$ for some $V' \in \text{Tr}(\lambda)$.

Let $G$ and $G'$ be as in the lemma, and let $H$ be a labeled graph on $t$ vertices.
We carry out the counting of induced $H$-s in $G \odot G'$,
by applying the law of total probability over the possible partitions of~$V(H)$
according to the $G$ component.
The reader may verify that this leads to an expression of the following form:
\begin{equation}
r(H,\;G \odot G') \;=\;
\sum\limits_{\lambda \in \Lambda(H)} \frac{(s)_{\ell(\lambda)}}{s^t} \cdot
p(H[\lambda], G) \sum\limits_{H' \in \Gamma(H, \lambda)} r(H',G') \;.
\label{rrr}
\end{equation}
Recall that the falling factorial $(s)_\ell = s(s-1)(s-2)...(s-\ell+1)$ is zero if $\ell > s$.
Together with~(\ref{ppp}), this implies that $r(H,\;G \odot G')$ is bilinear in
$\mathbf{p}_{\min(s,t)}(G)$ and $\mathbf{r}_t(G')$, as required.
\end{proof}

\begin{rmk*}
Let $G'=K_1$, the graph with a single vertex and no edges,
and suppose that $G$ doesn't contain loops.
Then Lemma~\ref{bilinear} establishes a useful relation between
$\mathbf{r}_t(G)$ and $\mathbf{p}_t(G)$, the profiles with and without replacements.
The inner sum in~(\ref{rrr}) reduces, in this case, to $0$ or $1$.
\end{rmk*}

Consider the linear map $f_t(G)$, defined by
$f_t(G)\mathbf{v} = B_{s,t}(\mathbf{p}_m(G),\mathbf{v})$.
Note that $f_t(G)$ is represented by a stochastic matrix,
and in non-degenerate cases it has a single stationary state vector $\mathbf{q}_t(G)$,
such that $f_t(G) \mathbf{q}_t(G) = \mathbf{q}_t(G)$.
This implies
\begin{equation}
\mathbf{q}_t(G) \;=\;
\lim\limits_{n \rightarrow \infty} \mathbf{r}_t\left(G^{\odot n}\right)\;.
\label{lim}
\end{equation}
Thus $\mathbf{q}_t(G)$ is the limiting $t$-profile of the nested blow-up of $G$,
or shortly the \emph{nested $t$-profile} of $G$.
Finding $\mathbf{q}_t(G)$ reduces to computing the Perron-Frobenius eigenvector of $f_t(G)$.
We also define an unlabeled variant, $Q(H,G) = q(H,G) \cdot |orbit(H)|$.

We now demonstrate the above notions on the graph $G = K_3 \otimes K_3$.
Here Lemma~\ref{bilinear} reads
$\mathbf{r}_4(G \odot\nobreak G') = B_{9,4}(\mathbf{p}_4(G),\mathbf{r}_4(G'))$,
which lets us compute $f_4(G)$ using~(\ref{rrr}).
Instead of $\mathbf{q}_4(G)$, we prefer to consider
the equivalent unlabeled nested $4$\nobreakdash-profile $\mathbf{Q}_4(G)$,
which is the eigenvector of another matrix $F_4(G)$, easily derived from $f_4(G)$.
This reduces the order of the matrix from $64$ to $11$,
which enables us to explicitly record it here.
The ordering of the basis is fixed to
$ K_4, A_4, T_4, S_4, M_4, C_4, Q_4, V_4, D_4, E_4, P_4 $ (See Table~\ref{exootable}).
$$ F_4(K_3 \otimes K_3) \;=\; \frac{1}{729} \left( \begin{array}{ccccccccccc}
 53 &   0 &  16 &  12 &  12 &  24 &  24 &   8 &  36 &   4 &  16 \\
  0 &  53 &  12 &  16 &  24 &  12 &   8 &  24 &   4 &  36 &  16 \\
112 &   0 &  53 &  48 &  32 &  64 &  68 &  32 &  88 &  16 &  48 \\
  0 & 112 &  48 &  53 &  64 &  32 &  32 &  68 &  16 &  88 &  48 \\
 84 &  24 &  48 &  48 &  45 &  64 &  60 &  40 &  72 &  32 &  52 \\
 24 &  84 &  48 &  48 &  64 &  45 &  40 &  60 &  32 &  72 &  52 \\
192 &  96 & 156 & 144 & 144 & 160 & 165 & 136 & 176 & 120 & 152 \\
 96 & 192 & 144 & 156 & 160 & 144 & 136 & 165 & 120 & 176 & 152 \\
 48 &  24 &  48 &  60 &  32 &  56 &  56 &  44 &  57 &  32 &  48 \\
 24 &  48 &  60 &  48 &  56 &  32 &  44 &  56 &  32 &  57 &  48 \\
 96 &  96 &  96 &  96 &  96 &  96 &  96 &  96 &  96 &  96 &  97
\end{array} \right)
$$
which implies
$$
\mathbf{Q}_4(K_3 \otimes K_3) \;=\; \frac{1}{728} \left( \begin{array}{ccccccccccc}
17 &  17 &  50 &  50 &  51 &  51 & 150 & 150 &  48 &  48 &  96
\end{array} \right)^T\;.
$$
We now proceed to the spectrum. Following Thomason,
we write down the relevant Fourier coefficients of each factor of the tensor products
in question.
\begin{center}
\begin{tabular}{l*{5}{c}}
\hline
$H$                 		& $K_4$ & $M_4$ & $C_4$ & $Q_4$ & $V_4$ \\
\hline
$\hat{r}(H,K_4)$ 			& -1/2  &  1/4  &  1/4  & -1/8  &  1/4  \\
$\hat{r}(H,M_4)$            &  1/2  &  1/4  &  1/4  &  1/8  &  1/4  \\
$\hat{q}(H,K_3 \otimes K_3)$& 18/91 &   0   &  9/91 &   0   &   0   \\
\hline
\end{tabular}
\end{center}
The desired limits in~(\ref{RK4A4}) and~(\ref{RP4}) then follow
by plugging these numbers into~(\ref{funcK4A4}) and~(\ref{funcP4})
by means of~(\ref{pointwisewise}) and~(\ref{lim}).

\section{Discussion}\label{discsect}

Propositions~\ref{IP4} and~\ref{iK4A4}
run new candidates for $I(P_4)$ and $i(K_4+A_4)$ respectively.
The constructed graphs have many symmetries, and in particular are vertex-transitive.
It is unclear whether an extremal construction for these problems has to be symmetric at all.
The authors have no particular reason to expect these graphs to be best possible,
even within the capability of the methods in use.
In this section we give a brief account of the state of these questions
for graphs with more than $4$ vertices, which we believe to support our skepticism.

Razborov's flag algebra calculus~\cite{razborov2007,razborov2013}
is the most powerful currently available method
for proving that some candidate construction is optimal.
However, as discussed in Section $4$ of~\cite{falgas2013},
it is not so good in dealing with nested constructions like ours.
This, too, can at least partly explain the current gaps between the bounds.

What is the inducibility of longer paths?
Let $P_t$ be a path on $t \geq 5$ vertices.
As shown by Exoo (see Section~\ref{constsect}), the nested blow up of $C_{t+1}$ yields
$ I(P_t) \geq t!/((t+1)^{t-1}-1) $.
An appropriate modification of the counting argument by Pippenger and Golumbic
(\cite{pippenger1975}, Theorem~$9$) gives $I(P_t) \leq t!/2(t-1)^{t-1}$.
In conclusion, the inducibility of the $t$-path is $\Theta\left(t^{3/2}\exp(-t)\right)$,
and the ratio between the two bounds is asymptotically $e^2/2$.

As for monochromatic cliques,
the asymptotic behavior of $i(K_t+A_t)$ remains an intriguing open problem,
related to bounds on the Ramsey number $R(t,t)$.
The best-known bounds for $t \geq 6$ are
$$ (2.18)^{-(1+o(1))t^2} \;\leq\; i(K_t+A_t) \;\leq\; 0.835 \cdot 2^{1-\binom{t}{2}} $$
by Conlon~\cite{conlon2012} and Thomason~\cite{jagger1996} respectively.
Conlon conjectures that
$2.18$ can be replaced with a lower constant, maybe even $\sqrt2$,
and remarks that the Erd\H{o}s conjecture may still be true within a constant factor.
R\"{o}dl, however, conjectures that
the upper bound can be improved at least exponentially in $t$~\cite{franek2002}.

Let's look closer at Thomason's construction.
In terms of the tensor product, it is given in~\cite{thomason1997} by blow-ups of
$G_t = K_4 \otimes M_4^{~\otimes~t-1}$.
Note that $M_4$ and $K_4$ are Cayley Graphs of~$\F_2^2$,
whose sets of generators are characterized by the functions
$m(x) = x_1x_2$ and $k(x) = x_1 + x_2 + x_1x_2$ respectively.
Indeed, Thomason's original representation of $G_t$ was as a Cayley graph of~$\F_2^{2t}$,
generated by the quadratic form
$q(x) = (x_1 + x_2 + x_1x_2) + x_3x_4 + ... + x_{2t-1}x_{2t}$ over~$\F_2$.
Suppose that $t$ is not divisible by $4$.
Then, by linear automorphisms of~$\F_2^{2t}$, $q(x)$~is further equivalent to
at least one of the following symmetric quadratic forms over~$\F_2$:
$$ s_1(x) = \sum\limits_{1 \leq i \leq j \leq 2t} x_i x_j \;,\;\;\;\;\;\;\;\;\;\;\;
s_2(x) = \sum\limits_{1 \leq i < j \leq 2t} x_i x_j \;.$$	
Note that $s_k(x)=1$ if and only if the Hamming weight of $x$ equals $k$~or~$k+1$ mod~$4$.
This can simplify our choice of generators.
For instance, $G_5$ can be generated by all vectors in $\F_2^{10}$ of Hamming weight
$\in\nobreak\{1,2,5,6,9,10\}$.

Franek and R\"{o}dl~\cite{franek1993}
introduced further counter-examples to the Erd\H{o}s conjecture,
using other Cayley graphs of $\F_2^n$
which are generated by sets of Hamming distances.
A computer search for such graphs provided the best-known upper bounds on $i(K_t+A_t)$ for
$6 \leq t \leq 8$~\cite{franek2002,deza2012,shen2012}.
For example, in $\F_2^{10}$ the weights set $\{0,2,5,6,9,10\}$ yields
$i(K_6+A_6) \leq 0.74444 \cdot 2^{-14}$, improving $0.76414 \cdot 2^{-14}$ obtained by $G_5$.
Note that the two graphs differ in about $1\%$ of the edges.
In general, these new constructions can be interpreted as
minor modifications of the ones derived from $s_1(x)$ and~$s_2(x)$.
It would be interesting to understand why this method generates better graphs,
and to find how its performance grows with~$t$.
Meanwhile, it indicates that our graph construction toolbox is still far from sufficient.

We close this note with a quick glance at $5$-vertex inducibilities.
Our current state of knowledge is collected in a table in Appendix~\ref{app1}.
Lower and upper bounds on the inducibility of all $34$ graphs with $5$ vertices are
listed there, together with a short description of the construction leading to the lower bound.
The quoted upper bounds were generated by
the excellent freeware \emph{Flagmatic}, created by Vaughan~\cite{flagmatic},
which has reduced such calculations to typing one line on the computer.
We used the method \texttt{GraphProblem(7,density=...).solve\_sdp()},
but didn't try to round the results to rational numbers.
Consequently, the upper bounds should only be viewed as well-established conjectures.

Some trends emerge in Table~\ref{exoo5table}.
In the first eight lines, the lower bound comes from blow-up of small graphs.
The choice of the blown-up graph is rather natural in all these cases.
Only in two of the cases are the upper and the lower bounds different,
and then, too, the difference is quite small.
In this view we suspect that the lower bound is the correct value.

The next five constructions are again either tight or at least plausibly so.
They each consist of two disjoint blow-ups of the same small graph,
with size-ratios optimized to $\alpha = 2+\sqrt{3}$.
This number comes up as the ratio $p:q$ that maximizes $pq^4+p^4q$ subject to $p+q=1$.
In all five cases the target graph is a connected graph plus an isolated vertex.
Moreover, the graph that we duplicate and blow up is
the best construction for the inducibility of the $4$-vertex component,
as in Table~\ref{exootable}.
The only exception is, again, $P_4$.
It is plausible that such relations between constructions carry on in larger graphs.

For the cycle $C_5$, nested blow-up has been long conjectured,
as discussed in Section~\ref{constsect}.

The remaining four cases seem to exhibit more complex behavior.
In two cases, our current constructions use random graphs.
One of them is tight and in the other one there is still a considerable gap,
which make us doubt its optimality.
It would be interesting to explore
the role of random constructions in the study of inducibility.
The challenge of derandomization suggests itself as well.
This may require the introduction of new machinery in the realm of constructions
beyond blow-ups, nesting and products.

Our best constructions for the two last cases,
the $5$-path and the self-complementary ``bull'' graph,
combine nested blow-up with tensor products.
This situation is similar to~$P_4$ and~$K_4+\nobreak A_4$, the heroes of this note.
Thus we believe these cases to be relevant as well to
the search for new interesting graph constructions.

\bibliographystyle{abbrv}
{ 
\bibliography{ind}
}

\newpage
\appendix
\section{The inducibility of 5-vertex graphs}
\label{app1}
\nobreak
\begin{table}[h]
\centering\renewcommand{\arraystretch}{2}\begin{tabu}{c c l l l l}
\tabucline[1pt]{-}\tabucline[1pt]{-}
$H,\;\overline{H}$ & orbit & FA bound & & Lower bound & Construction \\
\tabucline[1pt]{-}\tabucline[1pt]{-}
\pentagraph{}
\pentagraph{(1)--(2) (1)--(3) (2)--(3) (1)--(4) (2)--(4) (3)--(4) (1)--(5) (2)--(5) (3)--(5) (4)--(5) }
& 1
& 1
& = & 1 & $A_n$ \\ \hline
\pentagraph{(2)--(3) (2)--(4) (3)--(4) (1)--(5) }
\pentagraph{(1)--(3) (5)--(3) (1)--(4) (5)--(4) (1)--(2) (2)--(5) }
& 10
& 0.625
& = & 5/8 & $2 \times K_n$ \\ \hline
\pentagraph{(2)--(5) (1)--(4) }
\pentagraph{(3)--(2) (5)--(3) (3)--(4) (2)--(4) (5)--(4) (3)--(1) (2)--(1) (1)--(5) }
& 15
& 0.37037037
& = & 10/27 & $3 \times K_n$ \\ \hline
\pentagraph{(1)--(5) }
\pentagraph{(2)--(4) (2)--(3) (4)--(3) (2)--(5) (4)--(5) (3)--(5) (2)--(1) (4)--(1) (3)--(1) }
& 10
& 0.51367669
& $\geq$ &  0.5126953125 & $8 \times K_n$ \\ \hline
\pentagraph{(1)--(5) (2)--(3) (3)--(4) }
\pentagraph{(1)--(2) (2)--(3) (3)--(4) (4)--(5) (5)--(1) (1)--(4) (2)--(5) }
& 30
& 0.27777778
& = & 5/18 & $2 \times K_{n,n,n}$\\ \hline
\pentagraph{(5)--(1) (1)--(2) (2)--(4) (2)--(3) }
\pentagraph{(5)--(1) (1)--(2) (2)--(4) (2)--(3) (1)--(4) (4)--(5) }
& 60
& 0.192
& = & 24/125 & $C_5 \otimes A_n$ \\ \hline
\pentagraph{(1)--(5) (1)--(3) (5)--(3) }
\pentagraph{(4)--(2) (4)--(3) (2)--(3) (1)--(4) (2)--(1) (4)--(5) (2)--(5) }
& 10
& 0.3456
& = & 216/625 & $K_{3n} \cup A_{2n}$ \\ \hline
\pentagraph{(3)--(5) (3)--(1) }
\pentagraph{(4)--(2) (4)--(5) (2)--(5) (4)--(1) (2)--(1) (5)--(1) (4)--(3) (2)--(3) }
& 30
& 0.27886864
& $\geq$ & 0.2784 & $K_{n,n} \cup K_{2n,2n} \cup K_{2n,2n}$ \\
\tabucline[1pt]{-}
\pentagraph{(1)--(2) (1)--(5) (2)--(5) (1)--(4) (2)--(4) (5)--(4) }
\pentagraph{(3)--(2) (1)--(3) (3)--(4) (3)--(5) }
& 5
& 0.41666667
& = & 5/12 & $K_n \cup K_{\alpha n}$ \\ \hline
\pentagraph{(3)--(2) (1)--(3) (3)--(5) }
\pentagraph{(1)--(2) (1)--(5) (2)--(5) (1)--(4) (2)--(4) (5)--(4) (4)--(3) }
& 20
& 0.20833333
& = & 5/24 & $K_{n,n} \cup K_{\alpha n,\alpha n}$ \\ \hline
\pentagraph{(5)--(2) (2)--(4) (1)--(4) (5)--(1) }
\pentagraph{(1)--(2) (1)--(3) (2)--(3) (3)--(4) (3)--(5) (4)--(5) }
& 15
& 0.15625
& = & 5/32 & $K_{n,n} \cup K_{\alpha n,\alpha n}$ \\ \hline
\pentagraph{(1)--(5) (1)--(3) (5)--(3) (1)--(4) }
\pentagraph{(1)--(2) (2)--(3) (2)--(4) (3)--(4) (2)--(5) (4)--(5) }
& 60
& 0.1562855
& $\geq$ & 0.15625 & $(2 \times K_{n,n})^C \cup (2 \times K_{\alpha n,\alpha n})^C$ \\ \hline
\pentagraph{(1)--(2) (2)--(5) (1)--(4) (2)--(4) (5)--(4) }
\pentagraph{(3)--(2) (1)--(3) (3)--(4) (3)--(5) (1)--(5) }
& 30
& 0.24001625
& $\geq$ & 0.24 & $K_{n,n,n,n,n} \cup K_{\alpha n,\alpha n,\alpha n,\alpha n,\alpha n}$ \\
\tabucline[1pt]{-}
\pentagraph{(2)--(3) (5)--(4) (3)--(4) (1)--(5) (2)--(1) }
& 12
& 0.038462591
& $\geq$ & 1/26 $\approx$ 0.0384615 & $C_5^{\odot n}$ \\ \hline
\pentagraph{(1)--(2) (2)--(4) (4)--(5) (5)--(1) (2)--(3) }
\pentagraph{(1)--(2) (2)--(3) (2)--(4) (3)--(4) (1)--(5) }
& 60
& 0.25117348
& = & 15625/62208 & $G(n,n,5/6)$ \\ \hline
\pentagraph{(2)--(5) (1)--(4) (2)--(4) }
\pentagraph{(3)--(2) (5)--(3) (3)--(4) (5)--(4) (3)--(1) (2)--(1) (1)--(5) }
& 60
& 0.14470304
& $\geq$ & 0.133413966 & $G(n,0.3)$ \\ \hline
\pentagraph{(1)--(2) (2)--(3) (2)--(4) (3)--(4) (4)--(5) (1)--(5) }
\pentagraph{(1)--(2) (2)--(3) (3)--(4) (4)--(5) }
& 60
& 0.095475179
& $\geq$ & 1968/20995 $\approx$ 0.09373 & $(K_3 \otimes K_3 \otimes K_2)^{\odot n}$ \\ \hline
\pentagraph{(4)--(2) (4)--(3) (2)--(3) (2)--(1) (4)--(5) }
& 60
& 0.077634203
& $\geq$ & 813/11111 $\approx$ 0.07317 & $(K_3 \otimes K_3)^{\odot n}$ \\ \hline
\tabucline[1pt]{-}\tabucline[1pt]{-}
\end{tabu}
\caption{Inducibilities of graphs on 5 vertices}
\label{exoo5table}
\end{table}
{ \small
\begin{ntn*}
As usual, $K_n$ is a clique, $A_n$ is an anticlique, $C_n$ is an cycle
and $K_{n,n,...}$ is a complete multi-partite graph.
The graph products $\odot$ and $\otimes$ are as defined in Section~\ref{constsect}.
The disjoint union of $G$ and $H$ is denoted $G \cup H$,
and $k \times G$ is the union of $k$ copies of~$G$, i.e., $A_k \odot G$.
Let $\alpha = 2+\sqrt{3}$,
where ``$\alpha n$'' should be rounded to the closest integer.
The graph $G^C$ is the complement to $G$.
$G(n,p)$ is the Erd\H{o}s-R\'enyi random graph, with $n$ vertices and edge-probability $p$,
and $G(n,m,p)$ is, similarly, a random bipartite graph.
\end{ntn*}}

\end{document}